\documentclass[12pt]{amsart}

\usepackage{amsmath,amssymb,amsfonts,graphics,amsthm}
\usepackage{verbatim}
\usepackage{fullpage}
\usepackage{tikz}
\usepackage{mathabx}

\newcommand{\N}{\mathbb{N}}
\newcommand{\Z}{\mathbb{Z}}
\newcommand{\R}{\mathbb{R}}

\newcommand{\C}{\mathbb{C}}
\newcommand{\F}{\mathbb{F}}
\newcommand{\G}{\mathbb{G}}

\newcommand{\T}{\mathbb{T}}

\newcommand{\mc}{\mathcal}
\newcommand{\id}{\text{id}}

\newcommand{\Tr}{\text{Tr}}

\newcommand{\Pol}{\text{Pol}}

\newtheorem{thm}{Theorem}[section]
\newtheorem{cor}[thm]{Corollary}
\newtheorem{lem}[thm]{Lemma}
\newtheorem{prop}[thm]{Proposition}

\theoremstyle{definition}
\newtheorem{defn}[thm]{Definition}

\theoremstyle{remark}
\newtheorem{rem}[thm]{Remark}
\newtheorem{notat}[thm]{Notation}

\numberwithin{equation}{section}

\begin{document}

\title[Quantum Groups and Generalized Circular Elements]
{Quantum Groups and Generalized Circular Elements} 

\date{\today}

\author{Michael Brannan and Kay Kirkpatrick}

\address{Michael Brannan and Kay Kirkpatrick: Department of Mathematics, University  of Illinois at Urbana-Champaign, Urbana, IL 61801, USA}
\email{mbrannan@illinois.edu, kkirkpat@illinois.edu}
\urladdr{http://www.math.uiuc.edu/~mbrannan,  http://www.math.uiuc.edu/~kkirkpat}


\begin{abstract}
We show that with respect to the Haar state, the joint distributions of the generators of Van Daele and Wang's free orthogonal quantum groups are modeled by free families of generalized circular elements and semicircular elements in the large (quantum) dimension limit.  We also show that this class of quantum groups acts naturally as distributional symmetries of almost-periodic free Araki-Woods factors.
\end{abstract}
\maketitle

\section{Introduction} \label{section:intro}

There are intriguing connections between the representation theory of certain classes of compact matrix groups, and independent Gaussian structures in probability theory.  For instance, if one considers the $N^2$ matrix elements $\{u_{ij}\}_{1 \le i,j \le N}$ of the fundamental representation of the $N \times N$ orthogonal group $O_N = O_N(\R)$ on the Hilbert space $\C^N$, then it is well known that the joint moments of these variables with respect to the Haar probability measure are approximated by an independent and identically distributed, mean zero, variance $\frac1N$ family of real Gaussian random variables in the large $N$ limit (see for example \cite{DiFr87}) .  Intimately related to this asymptotic Gaussianity result is the celebrated theorem of Freedman \cite{Fr, DiFr87}, which says that an infinite sequence ${\bf x} = (x_n)_{n \in \N}$ of real-valued random variables is a conditionally independent centered Gaussian family with
common variance if and only if the sequence is \textit{rotatable}: i.e., for each $N \in \N$,  the joint distribution of the $N$-dimensional truncation ${\bf x}_N = (x_n)_{1 \le n \le N}$ of ${\bf x}$ is invariant under rotations by $O_N$.  

When one replaces the family of orthogonal groups $\{O_N\}_{N \in \N}$ by the unitary groups $\{U_N\}_{N \in \N}$, analogous results are known to hold where one replaces real Gaussian random variables by their complex-valued counterparts.  The key ingredient for the above results is a certain asymptotic orthonormality property for canonical generators (weighted Brauer diagrams, in fact) of the spaces of intertwiners between the tensor powers of the fundamental representations of these groups in the large rank limit.  This asymptotic feature of the representation theory is most concisely expressed via the so-called Weingarten calculus developed in \cite{Co03, CoSn06}, with origins in the pioneering work of Weingarten \cite{We} on the asymptotics of unitary matrix integrals.  For a broad treatment of probabilistic symmetries, we refer to the text \cite{Ka}.

Within the framework of operator algebras and non-commutative geometry, \textit{compact quantum groups} provide a vast and rich generalization of the theory of compact groups.  The operator algebraic theory of compact quantum groups was pioneered by Woronowicz (see \cite{Wo,Wo2} for instance) and has led recently to many interesting examples and developments in the theory of operator algebras.

One can ask non-commutative probabilistic questions about compact quantum groups, because every compact quantum group $\G$ admits a natural analogue of the Haar probability measure (the Haar state).  The last decade or so has seen a flurry of activity in this direction, particularly for free quantum groups and Voiculescu's free probability theory.  For instance, the free orthogonal quantum groups $O_N^+$ and free unitary quantum groups $U_N^+$ discovered by Wang \cite{Wa0} have interesting probabilistic statements that show a deep parallel with the aforementioned classical results for $O_N$ and $U_N$.  Most notable for our purposes are the results of Banica-Collins \cite{BaCo} and Curran \cite{Cu}.  Banica and Collins show that the rescaled matrix elements $\{\sqrt{N}u_{ij}\}_{1 \le i,j \le N}$ of the fundamental representation of $O_N^+$ (respectively $U_N^+$) converge in joint distribution to a freely independent family of standard--mean zero, variance one--semicircular (respectively circular) elements in a free group factor.  Curran provides a free probability analogue of Freedman's rotatability theorem:  An infinite sequence ${\bf x} = (x_n)_{n \in \N}$ of self-adjoint non-commutative random variables in a W$^\ast$-probability space $(M,\varphi)$ is \textit{quantum rotatable} if and only if there exists a W$^\ast$-subalgebra $B \subseteq M$ and a $\varphi$-preserving conditional expectation $E:M \to B$ such that ${\bf x}$ is an identically distributed family of mean zero semicircular elements that is free with amalgamation over $B$.  A similar result for $U_N^+$ is obtained in \cite{Cu}.  

In this paper, we consider similar non-commutative probabilistic questions for a broad class of compact quantum groups introduced by Van Daele and Wang \cite{VaWa} that generalizes the construction of $U_N^+$ and $O_N^+$ (which they called \textit{universal quantum groups}).  To define such a universal quantum group, one uses an invertible matrix $F \in \text{GL}_N(\C)$ to deform the defining algebraic relations for the quantum groups $O^+_N$ and $U^+_N$, yielding a pair of new compact matrix quantum groups called $O^+_F$ and $U^+_F$.  When $F = 1$ (the $N \times N$ identity matrix), we recover $O_N^+$ and $U_N^+$ as special cases (see Section \ref{section:freeorthogonal} for precise definitions). We follow recent literature conventions and refer to $O^+_F$ as \textit{free orthogonal quantum groups} and $U^+_F$  as \textit{free unitary quantum groups} (both with parameter matrix $F$).

The quantum groups $\G = O^+_F,U^+_F$ (at least when $F$ is not a multiple of a unitary matrix) are especially interesting because their Haar states are non-tracial and because the corresponding quantum group von Neumann algebras $L^\infty(\G)$ are known to be type III factors in many cases (see \cite{DeFrYa, VaVe}).  

The main result of this paper is that in this far more general (possibly non-tracial) setting, asymptotic freeness still emerges in the large rank limit.  In theorems \ref{thm:asymptotic-freeness-general}, \ref{thm:asymptotic-freeness-large-rank} and \ref{thm:asymptotic-freeness-large-qrank}, we show that the joint distribution of the (suitably rescaled) matrix elements of the fundamental representation of the quantum group $O^+_F$ can be approximated by a freely independent family of non-commutative random variables consisting of semicircular elements and Shlyakhtenko's \textit{generalized circular elements} \cite{Sh}, which is built in a natural way from the initial data $F \in \text{GL}_N(\C)$.  Generalized circular elements are non-tracial deformations of Voiculescu's circular elements, and they arise as canonical generators of free Araki-Woods factors \cite{Sh}.  Since free Araki-Woods factors are the natural non-tracial (or type III) deformations of the free group factors, our asymptotic freeness results
for $O^+_F$ provide a satisfactory generalization of the tracial asymptotic freeness results of \cite{BaCo}.  We also remark in Section \ref{sect:unitary} how similar asymptotic freeness results can be obtained for the free unitary quantum groups $U^+_F$. 

Our proofs of asymptotic freeness results in $O^+_F$ and $U^+_F$ follow the general outline of the earlier work of \cite{BaCo, CoSn06}, and we use a modified version of the Weingarten calculus for our situation.  In our case, the formulas become a bit more unwieldy, a consequence of the extra parameters that arise from the non-trivial matrix $F \in \text{GL}_N(\C)$.  On the other hand, there is one significant and interesting difference between our (non-tracial) setting and the earlier asymptotic (free) independence results on groups and quantum groups where the Haar state is tracial.  In the case of $O^+_F$ and $U^+_F$, $F \in GL_N(\C)$, we find that the error in approximation of joint moments by free variables is of order $O\Big((\text{Tr}(F^*F))^{-1}\Big)$, a bound that is in many cases much smaller than the traditional bound given by $O\Big(\frac{1}{N}\Big)$ in the classical case.  This fact allows us to observe, for example, asymptotic freeness results in a fixed dimension $N$, by considering families of quantum groups $\{O^+_F\}_{F \in \Lambda \subset \text{GL}_N(\C)}$ where the {\it quantum dimension} $\text{Tr}(F^*F)$ tends to infinity (cf. Theorem \ref{thm:asymptotic-freeness-large-qrank}).

Based on our non-tracial asymptotic freeness results described above, together with Curran's work on quantum rotatability \cite{Cu}, it now becomes natural to ask whether the free quantum groups $O^+_F$ act non-trivially on free Araki-Woods factor in a free-quasi-free state-preserving way.  In Section \ref{qsym},  we answer this question in the affirmative, and as a result we observe that almost-periodic free Araki-Woods factors admit a wealth of  quantum symmetries.   A future goal of the authors is to find a suitable ``type III'' version of Freedman's theorem adapted to $O^+_F$, $U^+_F$, and free Araki-Woods factors.  After a first version of this paper appeared, it was pointed out to the authors that the main result of Section \ref{qsym}, namely Theorem \ref{thm:invariance-canonicalF}, can also be obtained as a special case of a very general result of S. Vaes  \cite[Proposition 3.1]{Va05}.  

We finish this section with an outline of paper's organization.  Section \ref{section:prelim} discusses some preliminaries on quantum groups and free probability that are required.  Section \ref{section:freeorthogonal}
defines the quantum groups $O^+_F$ and $U^+_F$ and gives Weingarten-type formulas for  joint moments of the generators of $O^+_F$ with respect to the Haar state.  Section  \ref{section:freeness} gives various asymptotic freeness results for $O^+_F$, and includes a remark on how to extend our results on $O^+_F$
to some of their unitary counterparts $U^+_F$.  Finally Section \ref{qsym}
considers $O^+_F$ as quantum symmetries of almost-periodic free Araki-Woods factors.  This is achieved by associating to each $O^+_F$ a canonical free family of generalized circular elements whose joint distribution is invariant under quantum rotations by $O^+_F$.

\subsection*{Acknowledgements}  M.B. was partially supported by an NSERC Postdoctoral Fellowship. K.K. was partially supported by NSF Grant DMS-1106770 and NSF CAREER Award DMS-1254791.  The authors wish to thank Todd Kemp for encouraging the authors to collaborate.

\section{Preliminaries} \label{section:prelim}

In this section we briefly review some concepts from free probability theory and compact quantum group theory.  For more details, we refer the reader to \cite{NiSp} for free probability and \cite{Ti, Wo2} for quantum groups.

\subsection{Non-commutative probability spaces and free independence}

A \textit{non-commutative probability space (NCPS)} is a pair $(A,\varphi)$, where $A$ is
a unital C$^\ast$-algebra, and $\varphi: A \to \C$ is a state (i.e. a linear functional such that
$\varphi(1_A) = 1$ and $\varphi(a^*a) \ge 0$ for all $a \in A$).  Elements $a \in A$ are called \textit{random variables}.  Given a family of random variables $X = \{x_r\}_{r \in \Lambda} \subset (A,\varphi)$, the \textit{joint distribution} of $X$ is the collection of all {\it joint $\ast$-moments} \[\{\varphi(P((x_{r})_{r \in \Lambda}): P \in 
\C\langle t_r, t_r^*: r \in \Lambda \rangle \},\] where $\C\langle t_r, t_r^*: r \in \Lambda \rangle$ is the unital $\ast$-algebra of non-commutative polynomials in the variables $\{t_r\}_{r \in \Lambda}$, equipped with anti-linear involution $t_r \mapsto t_r^*$.  Given another family or random variables $Y = \{y_r\}_{r \in \Lambda}$ in a NCPS $(B,\psi)$, we say that $X$ and $Y$ are \textit{identically distributed} if $\varphi(P((x_{r})_{r \in \Lambda}) = \psi(P((y_{r})_{r \in \Lambda})$ for all $P \in \C\langle t_r, t_r^*: r \in \Lambda \rangle$.

Let $(A,\varphi)$ be a NCPS.  A family of $\ast$-subalgebras $\{A_r\}_{r \in \Lambda}$ of $A$ is said to be \textit{freely independent} (or simply \textit{free}) if the following condition holds: for any choice of indices $r(1) \ne r(2), r(2) \ne r(3), \ldots , r(k -1 ) \ne r(k) \in \Lambda$ and any choice of centered random variables variables $x_{r(j)} \in A_{r(j)}$ (i.e., $\varphi(x_{r(j)}) = 0$), we have the equality
\[\varphi(x_{r(1)}x_{r(2)} \ldots x_{r(k)}) = 0.\]   

A family of random variables $X = \{x_r\}_{r \in \Lambda} \subset (A,\varphi)$ is said to be \textit{free} if the family of unital $\ast$-subalgebras
\[\{A_r\}_{r \in \Lambda}; \qquad  A_r := \text{alg}\big( 1,x_r, x_r^*\big),\]
is free in the above sense.   
Let $S_\alpha = \{x_r^{(\alpha)}\}_{r \in \Lambda} \subset (A_\alpha,\varphi_\alpha)$ be a net of families of random variables and $S = \{x_r\}_{r \in \Lambda} \in (A,\varphi)$ be another family of random variables.  We say that \textit{$S_\alpha$ converges to $S$ (or $S_\alpha \to S$) in distribution} if, for
any non-commutative polynomial $P \in \C\langle X_r: r \in \Lambda \rangle$, we have
\[\lim_{\alpha} \varphi_\alpha(P(S_\alpha)) = \varphi(P(S)).\] 


\subsection{Fock spaces, semicircular elements, and  generalized circular elements} \label{section:FAW}

Let $H$ be a complex Hilbert space.  The \textit{full fock space} is the Hilbert space 
\[
\mc F(H) = \bigoplus_{n =0}^\infty H^{\otimes n}, 
\] 
where we put $H^{\otimes 0}: = \C \Omega$, where $\Omega$ is a fixed unit vector, called the {\it vacuum vector}.   The {\it vacuum expectation} is the state $\varphi_{\Omega}: \mc B(\mc F(H)) \to \C$ given by $\varphi_\Omega(x)  = \langle  \Omega |x\Omega \rangle$, $x \in \mc B(\mc F(H))$.

For each $\xi \in H$, we define the \textit{left creation operator} $\ell(\xi) \in \mc B(\mc F(H))$
by  
\[
\ell(\xi) \Omega = \xi, \quad \ell(\xi) \eta = \xi \otimes \eta \qquad (\eta \in H^{\otimes n}, \ n \ge 1).
\]  Note that $\|\ell(\xi)\|_{\mc B(\mc F(H))} = \|\xi\|_H$.  Given a NCPS $(A,\varphi)$, a {\it (standard) semicircular element} is a self-adjoint random variable $x \in A$ with the same distribution as $s(\xi):= \ell(\xi)+\ell(\xi)^* \in (\mc B (\mc F(H)), \varphi_\Omega)$, where $\xi \in H$ is a unit vector.  Given $\alpha, \beta \in \R^+$, a random variable $x \in (A,\varphi)$ is called an {\it  $(\alpha, \beta)$-generalized circular element} if it has the same distribution as the element $\alpha \ell(\xi) + \beta \ell(\eta)^* \in (\mc B (\mc F(H)), \varphi_\Omega)$, where $\xi,\eta$ are orthonormal vectors in $H$.  One can readily verify that for an $(\alpha, \beta)$-generalized circular element $x$, one has
\[\varphi(x^*x) = \alpha^2 \quad \& \quad \varphi(xx^*) = \beta^2, \] and this information
completely determines the $\ast$-moments of $x$ with respect to $\varphi$.  We will call the numbers $\alpha^2$ and $\beta^2$ the \text{left and right variances} of $x$, respectively.  

Next, we want state a well known theorem which gives a combinatorial characterization of the joint distribution of a free semicircular or generalized circular family.  To do this, we first need some notation concerning non-crossing partitions that will be used below and throughout the remainder of the paper.

\begin{notat}
Let $k \in \N$ and denote by $[k]$ the ordered set $\{1, \ldots, k\}$. 
\begin{enumerate}
\item The lattice of (non-crossing) partitions of $[k]$ will be denoted by $\mc P(k)$ (resp. $NC(k)$), and the standard partial order on both lattices will be denoted by $\le$.  
\item If $\pi \in \mc P(k)$ partitions $[k]$ into $r$ disjoint, non-empty subsets $\mc V_1,\ldots, \mc V_r$ (called {\it blocks}), we write $|\pi| = r$ and say that {\it $\pi$ has $r$ blocks}.   
\item Given a function $i:[k] \to \Lambda$, we denote by $\ker i$ the element of $\mathcal P(k)$ whose blocks are the equivalence classes of the
relation $$s \sim_{\ker i} t \iff i(s) = i(t).$$ Note that
if $\pi  \in \mathcal P(k)$, then $\pi \le \ker i$ is equivalent to
the condition that whenever $s$ and $t$ are
in the same block of $\pi$, $i(s)$ must equal $i(t)$ (i.e the
 function $i:[k] \to \Lambda$ is constant on the blocks of $\pi$).  
\item Elements of $\mc P(k)$ which partition $[k]$ into subsets with exactly two elements are called {\it pairings} and the set of pairings of $[k]$ is denoted by $\mc P_2(k)$.  We also write $NC_2(k) = \mc P_2(k) \cap NC(k)$.  If $k$ is odd, we of course have $\mc P_2(k) =NC_2(k)= \emptyset$.
\item Given $\pi \in \mc P_2(k)$ and $s,t \in [k]$, we will always write  $(s,t) \in \pi$ if $\{s,t\}$ is a block of $\pi$ and $s < t$.  
\item Let $\epsilon:[k] \to \{1,\ast\}$ be a function.  We let $NC_2^\epsilon(k) \subset NC_2(k)$ be the subset of all non-crossing pairings such that \[\forall (s,t) \in \pi \implies \epsilon(s) \ne \epsilon(t).\]
\end{enumerate}
\end{notat}

\begin{thm}[See Chapters 7 and 15 of \cite{NiSp}] \label{thm:mc-formulas}
Let $X = (x_r)_{r \in \Lambda}$ be a family of random variables in a NCPS $(A,\varphi)$.
\begin{enumerate}
\item If $x_r = x_r^*$ for each $r \in \Lambda$, then $X$ is a free family of standard semicircular variables if and only if 
\begin{align*}\varphi(x_{r(1)}\ldots x_{r(k)}) &=\sum_{\substack{\pi \in NC_2(k) \\
\ker r \ge \pi}} \prod_{(s,t) \in \pi} \varphi(x_{r(s)}x_{r(t)}) \\
&= |\{\pi \in NC_2(k): \ker r \ge \pi\}| \qquad(k \in \N, \ r:[k] \to \Lambda).
\end{align*}
\item Let $(\alpha_r, \beta_r)_{r \in \Lambda} \subset \R^+\times \R^+$.  Then $X$ is a free family of $(\alpha_r,\beta_r)$-generalized circular elements if and only if 
\[\varphi(x_{r(1)}^{\epsilon(1)}\ldots x_{r(k)}^{\epsilon(k)}) = \sum_{\substack{\pi \in NC_2^{\epsilon}(k) \\
\ker r \ge \pi}} \prod_{(s,t) \in \pi} \varphi(x_{r(s)}^{\epsilon(s)}x_{r(t)}^{\epsilon(t)}) \qquad (k \in \N, \ r:[k] \to \Lambda, \ \epsilon:[k] \to \{1,\ast\}).\]   
\end{enumerate}
\end{thm}

\subsection{Free Araki-Woods factors and generalized circular elements}

Let $H_\R$ be a real separable Hilbert space  and let $(U_t)$ be an orthogonal representation of $\R$ on $H_\R$. Let $H = H_\R \otimes_\R \C$ be
the complexified Hilbert space.  If $A$ is the infinitesimal generator of (the extension of) $U_t$ on $H$  (i.e., $U_t = A^{it}$), then it follows that the map  $j:H_\R \hookrightarrow H$ defined
by $j(\xi) = \big(\frac{2}{A^{-1}+1}\big)^{1/2}\xi$ is an isometric embedding of $H_\R$ into $H$  \cite{Sh}. Let $K_\R = j(H_\R)$, then we have $K_\R \cap iK_\R = \{0\}$ and
$K_\R + iK_\R$ is dense in $H$. The {\it free Araki-Woods factor} is the von Neumann algebra 
\[\Gamma(H_\R, U_t)''  = W^*(\ell(\xi) +\ell(\xi)^*: \xi \in K_\R) \subseteq \mc B(\mc F(H)).\] 
The restriction of the vacuum expectation $\varphi_\Omega = \langle \Omega |\cdot \Omega\rangle$ on $\mc B(\mc F(H))$ to $\Gamma(H_\R, U_t)''$ is always a faithful normal state, and turns $(\Gamma(H_\R, U_t)'', \varphi_\Omega)$ into a non-commutative probability space.

We recall from \cite{Sh} that $U_t$ is the trivial representation if and only if $\Gamma(H_\R, U_t)'' \cong L(\F_{\dim H_\R})$, the von Neumann algebra generated by the left regular representation of the free group on $\dim H_\R$ generators.  Otherwise, $\Gamma(H_\R, U_t)''$ is a type III factor.   

Free Araki-Woods factors arise naturally when one considers free families of generalized circular elements that we introduced earlier.  More precisely, we have the following theorem, which follows easily from the results in \cite[Section 6]{Sh}.

\begin{thm}[\cite{Sh}] \label{thm:repfromgenerators}
Let $X = (x_r)_{r \in \Lambda}$ be a free family of $(\alpha_r,\beta_r)$-generalized circular elements in a non-commutative probability space $(A,\varphi)$ and let $0 < \lambda_r = \min\{\alpha \beta^{-1}, \beta \alpha^{-1}\} \le 1$.  Then there is a state-preserving $\ast$-isomorphism $(W^*(X), \varphi) \cong (\Gamma(H_\R, U_t)'', \varphi_\Omega)$, where 
$U_t$ is the almost-periodic orthogonal representation acting on the Hilbert space $H_\R = \bigoplus_{r \in \Lambda} \R^2$ given by \[U_t  = \bigoplus_{r \in \Lambda} R_{\lambda_r}(t) \quad \text{where} \quad R_{\lambda_r}(t) = \left[\begin{matrix} 
\cos(t\log {\lambda_r}) & -\sin(t\log {\lambda_r})\\
\sin (t \log {\lambda_r}) & \cos (t \log {\lambda_r})
\end{matrix}\right]. \] 
Moreover, every free Araki-Woods factor $\Gamma(H_\R, U_t)''$ arising from an almost-periodic representation $U_t$ arises in this fashion.
\end{thm}

\subsection{Compact quantum groups}

A compact quantum group is a pair $\G = (C(\G),\Delta )$ where $C(\G)$ is a unital
$C^*$-algebra and $\Delta : C(\G)\to C(\G)\otimes C(\G)$ is a unital $*$-homomorphism
satisfying
\begin{align*}
  &(\iota \otimes \Delta )\Delta = (\Delta \otimes \iota)\Delta \quad\mbox{   (coassociativity)}\\
  &[\Delta (C(\G))(1\otimes C(\G))]=[\Delta (C(\G)) (C(\G)\otimes 1)]=C(\G)\otimes C(\G) \quad \mbox{
    (non-degeneracy)},
\end{align*}
where $[S]$ denotes the norm-closed linear span of a subset $S\subset C(\G) \otimes
C(\G)$. Here and in the rest of the paper,  the symbol $\otimes$ will denote the minimal tensor product of C$^\ast$-algebras, $\overline{\otimes}$ will denote the spatial tensor product of von Neumann algebras, and $\odot$ will denote the algebraic tensor product of complex associative algebras.  The homomorphism $\Delta$ is called a \textit{coproduct}.  


For any compact quantum group $\G = (C(\G),\Delta)$, there exists a unique
\textit{Haar state} $h_{\G}: C(\G)\to \mathbb{C}$ which satisfies the following left and
right $\Delta$-invariance property, for all $a\in C(\G)$:
\begin{equation} \label{Haar} (h_{\G}\otimes \iota)\Delta (a)=(\iota\otimes h_{\G})\Delta
  (a)=h_{\G}(a)1.
\end{equation}
Note that in general $h = h_\G$ is not faithful on $C(\G)$.  In any
case, we can construct a GNS representation $\pi_h:C(\G) \to \mc B(L^2(\G))$, where
$L^2(\G)$ is the Hilbert space obtained by separation and completion of $C(\G)$ with
respect to the sesquilinear form $\langle a |b \rangle = h(a^*b),$ and $\pi_h$
is the natural extension to $L^2(\G)$ of the left multiplication action of $C(\G)$
on itself. 
 The \textit{von Neumann
  algebra of $\G$} is given by
\[L^\infty(\G) = \pi_h(C(\G))'' \subseteq \mc B(L^2(\G)).\] We note that $\Delta_r$
extends to an injective normal $\ast$-homomorphism $\Delta_r:L^\infty(\G)\to
L^\infty(\G) \overline{\otimes} L^\infty(\G)$, and the Haar state on $C(\G)$
lifts to a faithful normal $\Delta_r$-invariant state on $L^\infty(\G)$.

Let $H$ be a finite dimensional Hilbert space and $U\in \mc B(H)\otimes C(\G)$ be invertible (unitary). Then $U$ is called a \textit{(unitary)
  representation} of $\G$ if, following the leg numbering convention,
\begin{equation} \label{rep} (\iota \otimes \Delta) U=U_{12}U_{13}.
\end{equation}  
If we fix an orthonormal basis of $H$, we can
identify $U$ with an invertible matrix $U = [u_{ij}] \in M_N(C(\G))$ and \eqref{rep}
means exactly that
\[
\Delta(u_{ij}) = \sum_{k=1}^N u_{ik} \otimes u_{kj} \qquad (1 \le i,j \le N).
\]
Of course the unit $1 \in C(\G)$ is always a representation of $\G$, called the
\textit{trivial representation}.

Let $U \in \mc B(H_1)\otimes C(\G)$ and $V \in \mc B(H_2)\otimes C(\G)$ be two
representations of $\G$.  An \textit{intertwiner} between $U$ and
$V$ is a bounded linear map $T:H_1 \to H_2$ such that $(T \otimes \iota )U = V
(T \otimes \iota)$.  The Banach space of all such intertwiners is denoted by
$\text{Hom}_\G(U, V)$.  When $U = 1$ is the trivial representation, we write $\text{Hom}_\G(U, V) = \text{Fix}(V) \subset H_2$, and call $\text{Fix}(V)$ the \textit{space of fixed vectors} for $V$.  If there exists an invertible (unitary) intertwiner
between $U$ and $V$, they are said to be \textit{(unitarily) equivalent}. A
representation is said to be irreducible if its only self-intertwiners are the
scalar multiples of the identity map.  It is known that each irreducible
representation of $\G$ is finite dimensional and every finite dimensional
representation is equivalent to a unitary representation. In addition, every
unitary representation is unitarily equivalent to a direct sum of irreducible
representations.

A compact quantum group $\G$ is called a \textit{compact matrix quantum group}
if there exists a finite dimensional unitary representation $U = [u_{ij}] \in
M_N(C(\G))$ whose matrix elements generate $C(\G)$ as a C$^\ast$-algebra.  Such a
representation $U$ is called a \textit{fundamental representation} of $\G$.  In
this case, we note that the Haar state $h$ is faithful when restricted to the dense unital $\ast$-subalgebra $\Pol(\G) \subseteq C(\G)$ generated by $\{u_{ij}\}_{1 \le  i,j \le N}$.  


\section{The free quantum groups $O_F^+$ and $U^+_F$} \label{section:freeorthogonal}

In this section we recall the definition of the free orthogonal and unitary quantum groups $O^+_F$ and $U^+_F$, introduced by Van Daele and Wang in \cite{VaWa}.  

\begin{notat}
Given a complex $\ast$-algebra $A$ and a matrix $X = [x_{ij}] \in M_N(A)$, we denote by $\bar X$ the matrix $[x_{ij}^*] \in M_N(A)$.
\end{notat}

\begin{defn}[\cite{VaWa}] Let $N \ge 2$ be an integer and let $F \in \text{GL}_N(\C)$.
\begin{enumerate}
\item   The \textit{free unitary quantum group} $U_F^+$ (with parameter matrix $F$) is the compact quantum group given by the universal C$^\ast$-algebra 
\begin{align} \label{eqn:defining} C(U^+_F) = C^*\big(\{v_{ij}\}_{1 \le i,j \le N} \ | \ U = [v_{ij}] \text{ is unitary } \& \ F \bar U F^{-1} \text{ is unitary } \big),\end{align} together with coproduct $\Delta:C(U^+_F) \to C(U^+_F) \otimes C(U_F^+)$ given by
\[\Delta(v_{ij}) = \sum_{k=1}^N v_{ik} \otimes v_{kj} \qquad (1 \le i,j \le N).\]
\item Let $c = \pm 1$ and assume that $F \bar F = c 1$.  The \textit{free orthogonal quantum group} $O_F^+$ (with parameter matrix $F$) is the compact quantum group given by the universal C$^\ast$-algebra 
\begin{align} \label{eqn:defining} C(O^+_F) = C^*\big(\{u_{ij}\}_{1 \le i,j \le N} \ | \ U = [u_{ij}] \text{ is unitary } \& \ U = F \bar U F^{-1}\big),\end{align} together with coproduct $\Delta:C(O^+_F) \to C(O^+_F) \otimes C(O_F^+)$ given by
\[\Delta(u_{ij}) = \sum_{k=1}^N u_{ik} \otimes u_{kj} \qquad (1 \le i,j \le N).\]
\end{enumerate}
\end{defn}

\begin{rem}
The coproduct $\Delta$ is defined so that the matrix of generators $U = [u_{ij}]$ is always a fundamental representation of the compact matrix quantum groups $U^+_F$ and $O_F^+$, respectively.  
\end{rem}
\begin{rem}
Note that the above definition for $O^+_F$ makes sense for any $F\in \text{GL}_N(\C)$.  The additional condition $F \bar F = \pm 1$ is equivalent to the requirement that $U$ is always an irreducible representation of $O^+_F$.  Indeed, Banica \cite{Ba0} showed that $U$ is irreducible if and only if $F\bar F  = \pm \lambda 1$  ($\lambda > 0$), and moreover we clearly have $O_F^+ = O_{\lambda^{-1/2}F}^+$. 

We remark that for our asymptotic freeness results, our assumption that $F\bar F = \pm I$ is not a major restriction.  Indeed, by a result of Wang \cite[Section 6]{Wa02}, $O^+_F$ for generic $F \in \text{GL}_N(\C)$ can be decomposed into a free product of finitely many quantum groups $O^+_{F_i}$ and $U^+_{P_k}$ with $F_i, P_k$ invertible matrices and $F_i \bar{F_i} = \pm 1$.   
\end{rem}

For the remainder of the paper we will deal mostly with the free orthogonal quantum groups $O^+_F$.  Later on in Section \ref{sect:unitary} we indicate how to extend some of our orthogonal results to the unitary case.

\subsection{Canonical $F$-matrices for $O^+_F$} \label{section:canonicalF}

Let $c \in \{\pm 1\}$ and let $F \in \text{GL}_N(\C)$ be such that $F \bar F = c 1$.    In \cite{BiDeVa}, it is shown that if $c=1$, then there is an integer $0 \le k \le N/2$, a non-decreasing sequence $\rho = (\rho_i)_{i=1}^k \in (0,1)^k$, and a unitary $w \in U_N$ such that 
\begin{align} \label{eqn:canonicalF_c=1}
F_\rho^{(+1)}:=w^tFw = \left(\begin{matrix}0 & D_k(\rho) &0 \\
D_k(\rho)^{-1} &0 &0 \\
0 & 0 & 1_{N-2k}\end{matrix}\right),
\end{align}
where $D_k(\rho)$ denotes the $k \times k$ diagonal matrix with diagonal entries given by the $k$-tuple $\rho$.

On the other hand if $c= -1$, then by \cite{BiDeVa} we must have $N = 2k$ is even and there exists a non-decreasing sequence $\rho = (\rho_i)_{i=1}^k \in (0,1]^k$ and a unitary $w \in U_N$ such that 
\begin{align} \label{eqn:canonicalF_c=-1}
F_\rho^{(-1)}:=w^tFw = \left(\begin{matrix}0 & D_k(\rho)  \\
-D_k(\rho)^{-1} &0 \end{matrix}\right).
\end{align}
 
\begin{rem}
Note that the Kac type quantum groups $O_N^+$ correspond to the case $F = 1_N$, which is the canonical deformation matrix $F_\rho^{(+1)}$ with $k = 0$.
\end{rem}

According to \cite{BiDeVa}, given two matrices $F_i \in GL_{N_i}(\C)$ such that $F_i\bar F_i = c_i1$, the two free orthogonal quantum groups $O^+_{F_1}$ and $O^+_{F_2}$ are isomorphic if and only if $N_1 = N_2$, $c_1 = c_2$, and $F_2 = vF_1v^t$ for some unitary matrix $v \in U_{N_1}$.  The corresponding equivalence relation on such matrices has fundamental domain given by all matrices of the form $F^{(\pm 1)}_\rho$.  As a consequence, we call such matrices $F^{(\pm 1)}_\rho$ \textit{canonical} $F$-matrices.  The canonical $F$-matrices yield the most natural coordinate system in which to represent the isomorphism equivalence class of any given $O^+_F$.  
\subsection{Integration over $O^+_F$} \label{section:Weingarten}

In this section, we consider the problem of evaluating arbitrary monomials in the generators $\{u_{ij}\}_{1 \le i,j \le N}$ of $C(O^+_F)$ with respect to the Haar state $h_{O^+_F}$.

\begin{notat}
Fix an orthonormal basis $\{e_i\}_{i=1}^N$ for $\C^N$ and $F \in GL_N(\C)$.  Define \[\xi = \sum_{i=1}^N e_i \otimes e_i \quad \text{and} \quad \xi^F = (\id \otimes F)\xi = \sum_{i=1}^N e_i \otimes Fe_i.
\]  
For each $l \in \N$, $\pi \in NC_{2}(2l)$, and $i:[2l] \to [N]$ define \[\delta_\pi^F(i) = \prod_{(s,t) \in \pi} F_{i(t)i(s)},\] and put \[\xi^F_\pi = \sum_{i:[2l]\to [N]} \delta_\pi^F(i) e_{i(1)} \otimes e_{i(2)} \otimes \ldots \otimes e_{i(2l)} \in (\C^N)^{\otimes 2l} .\]  
\end{notat}
For the  purposes of integrating monomials over $O^+_F$ with respect to the Haar state, we are interested in the \textit{$l$-th tensor power} of the fundamental representation $U = [u_{ij}]$ of $O^+_F$, \[U^{\otop l}:= [u_{i(1)j(1)} \ldots u_{i(l)j(l)}] \in \mc B((\C^N)^{\otimes l}) \otimes C(O^+_F).\]  $U^{\otop l}$ is evidently a representation of the quantum group $O^+_F$, and the following theorem of Banica describes the space of fixed vectors of these higher tensor powers of $U$. 
\begin{thm}[\cite{Ba0}] \label{thm:basisofFix}
Let $N \ge 2$, $c \in \{\pm 1\}$ and $F \in GL_N(\C)$ be such that $F\bar F = c1$.  Then for each $l \in \N$, 
\[Fix(U^{\otop 2l+1}) = \{0\},\]  
and \[Fix(U^{\otop 2l}) = span\{\xi_\pi^F: \pi \in NC_2(2l)\}.\]  Moreover, $\{\xi_\pi^F\}_\pi$ is  a linear basis for $Fix(U^{\otop 2l})$. 
\end{thm}

With the preceding theorem at hand, we now use the Weingarten calculus to describe the Haar state on $O^+_F$ in terms of the Gram matrices associated to the bases $\{\xi^F_\pi\}_{\pi \in NC_2(2l)}$ of $\text{Fix}(U^{\otop 2l})$. For each $l \in \N$, define an $|NC_{2}(2l)| \times |NC_{2}(2l)|$ matrix $G_{2l,F} = [G_{2l,F}(\pi,\sigma)]_{\pi,\sigma \in NC_2(2l)}$ by \[G_{2l,F}(\pi,\sigma) = \langle \xi^F_\pi|\xi^F_\sigma \rangle \qquad (\pi,\sigma \in NC_2(2l)).\]

\begin{thm} \label{thm:grammatrix}
Let $N \ge 2$, $c \in \{\pm 1\}$ and $F \in GL_N(\C)$ be such that $F\bar F = c1$.  Set  $N_F:=Tr(F^*F)$.   Then for any $l \ge 1$, $G_{2l,F}$ is an invertible matrix and 
\[
G_{2l,F}(\pi,\sigma) = c^{l+|\pi \vee \sigma |}N_F^{|\pi \vee \sigma|} \qquad (\pi,\sigma \in NC_2(2l)), 
\] 
where $\pi\vee\sigma$ denotes the join of $\pi$ and $\sigma$ in the lattice $\mc P(2l)$.
\end{thm}

\begin{proof}
The first assertion follows from Theorem \ref{thm:basisofFix}.  For the second assertion, fix $\pi,\sigma \in NC_2(2l)$ and let $\pi \vee \sigma = \{\mc V_1, \mc V_2, \ldots, \mc V_{m}\}$, where each $\mc V_a$ is a block of $\pi\vee \sigma$.  Then we have 

\[G_{2l,F}(\pi,\sigma) = \sum_{i:[2l] \to [N]} \overline{\delta_\pi^F(i)} \delta_\sigma^F(i) =  \prod_{a=1}^m \Big(\sum_{i:\mc V_a \to [N]} \overline{\delta_{\pi|_{\mc V_a}}^F(i)}\delta_{\sigma|_{\mc V_a}}^F(i)\Big). \]  From the above equation we see that $G_{2l,F}(\pi,\sigma)$ is a multiplicative function of the blocks of $\pi \vee \sigma$, and therefore it suffices to prove the theorem when $\pi \vee \sigma = 1_{2l}$.  

To prove this special case of the theorem, we proceed by induction on $l$:  If $2l=2$, then $G_{2l,F} = \langle \xi^F|\xi^F \rangle  = Tr(F^*F) = N_F = c^2N_F$.  Now assume $l \ge 2$ and that the desired result is true for all $\pi', \sigma' \in NC_2(2l-2)$ with $\pi' \vee \sigma'=1_{2l-2}$.   Fix $\pi,\sigma \in NC_2(2l)$ such that $\pi \vee \sigma = 1_{2l}$.
Since $\pi$ is non-crossing, we can fix an interval $\{r,r+1\}$ in $\pi$ and let $\{a, r\}$, $\{b,r+1\}$ be the corresponding (unordered) pairs of $\sigma$ that connect to $r$ and $r+1$.  (Note that $\sigma$ does not pair $r$ and $r+1$ because $|\pi \vee \sigma| = 1$ and $l \ge 2$.)  Now let $\pi' \in NC_2(2l-2)$ be the pairing obtained by deleting the block $\{r,r+1\}$ from $\pi$ and let $\sigma' \in NC_2(2l-2)$ be the pairing obtained by deleting the points $r,r+1$ from $\sigma$ and pairing $a$ and $b$.  Note that by construction, $\pi' \vee \sigma' = 1_{2l-2}$.
       
Using the readily verified identities 
\[c\xi^F = (\iota \otimes (\xi^F)^* \otimes \iota)(\xi^F \otimes \xi^F) = ((\xi^F)^* \otimes \iota) (\iota \otimes \xi^F \otimes \iota)\xi^F = (\iota \otimes (\xi^F)^*) (\iota \otimes \xi^F \otimes \iota)\xi^F ,\]
it easily follows that $G_{2l,F}(\pi,\sigma) =  \langle \xi^F_\pi|\xi^F_\sigma \rangle = c\langle \xi^F_{\pi'}|\xi^F_{\sigma'} \rangle$.  We then have from our induction assumption that 
\[G_{2l,F}(\pi,\sigma) = c (c^{l-1+|\pi' \vee \sigma'|}N_F) = c^{l+1}N_F. \]
\end{proof}

For each $l \in \N$, denote by $W_{2l,F}$ the matrix inverse of $G_{2l,F}$.  In the following theorem we give a Weingarten-type moment formula for the Haar state on $O^+_F$ (compare with \cite{BaCo, BaCoZJ}).  

\begin{thm} \label{thm:moments}
For each pair of multi-indices $i,j:[l] \to [N]$, we have 
\[h_{O^+_F}(u_{i(1)j(1)}u_{i(2)j(2)}\ldots u_{i(l)j(l)}) = 0\] if $l$ is odd, and otherwise 
\[h_{O^+_F}(u_{i(1)j(1)}u_{i(2)j(2)}\ldots u_{i(l)j(l)}) = \sum_{\pi,\sigma \in NC_2(l)} W_{l,F}(\pi,\sigma) \overline{\delta^F_\pi(j)} \delta^F_\sigma(i). \]
\end{thm}

\begin{proof}
We use the fact that if $V \in \mc B(H) \otimes C(\G)$ is a unitary representation of a compact quantum group $\G$ with Haar state $h$, then $P_V = (\id \otimes h)(V)$ is the orthogonal projection onto $\text{Fix}(V)$.  Using this fact, the quantity we are interested in is the $(i,j)$-th  matrix element of the projection $P_{U^{\otop l}}$.  Since $P_{U^{\otop l}} = 0$ when $l$ is odd (by Theorem \ref{thm:basisofFix}), the first equality is immediate.

For the second equality, assume $l \ge 2$ is even.  Let $\{\xi^F_\pi\}_{\pi \in NC_2(l)}$ be the basis for $\text{Fix}(U^{\otop l})$ from Theorem \ref{thm:basisofFix} and define a new family  $\{\tilde \xi^F_\pi\}_{\pi \in NC_2(l)} \subset \text{Fix}(U^{\otop l})$ by \[\tilde \xi^F_\pi = \sum_{\sigma \in NC_2(l)} W^{1/2}_{l,F}(\pi,\sigma) \xi^F_\sigma,\] where $W^{1/2}_{l,F}$ 
is the matrix square root of $W_{l,F}$. Then $\{\tilde \xi^F_\pi\}_{\pi \in NC_2(l)}$ is an orthonormal basis for $\text{Fix}(U^{\otop l})$ by Theorem \ref{thm:grammatrix} and $P_{U^{\otop l}} = \sum_{\pi \in NC_2(l)} |\tilde\xi^F_\pi \rangle\langle \tilde \xi^F_\pi|$.  
Therefore 
\begin{align*}
h_{O^+_F}(u_{i(1)j(1)}u_{i(2)j(2)}\ldots u_{i(l)j(l)}) &= \langle e_i| P_{U^{\otop l}}e_j\rangle \\
& =  \sum_{\rho \in NC_2(l)} \langle \tilde\xi^F_\rho|e_j\rangle \langle e_i| \tilde \xi^F_\rho \rangle \\
& =  \sum_{\pi,\sigma, \rho \in NC_2(l)} W^{1/2}_{l,F}(\rho,\pi)\langle \xi^F_\pi|e_j\rangle W^{1/2}_{l,F}(\rho, \sigma) \langle e_i| \xi^F_\sigma\rangle \\
&= \sum_{\pi,\sigma \in NC_2(l)} W_{l,F}(\pi,\sigma) \langle \xi^F_\pi|e_j\rangle \langle e_i| \xi^F_\sigma\rangle\\
&=  \sum_{\pi,\sigma \in NC_2(l)} W_{l,F}(\pi,\sigma) \overline{\delta^F_\pi(j)} \delta^F_\sigma(i).
\end{align*}   
\end{proof}

\subsection{Integrating $\ast$-monomials over $O_F^+$} \label{section:integrate}  Note that the defining relations for the generators $\{u_{ij}\}_{1 \le i,j \le N}$ of $(C(O_F^+),h_{O_F^+})$ imply that the family $\{u_{ij}\}_{1 \le i,j \le N}$ is self-adjoint.  Moreover, taking the deformation matrix $F$ to be in canonical form, as defined in Section \ref{section:canonicalF}, we can write
\[F = F_\rho^{(c)}= \left(\begin{matrix}0 & D_k(\rho) &0 \\
cD_k(\rho)^{-1} &0 &0 \\
0 & 0 & 1_{N-2k}\end{matrix}\right). \] 
where $D_k(\rho)$ denotes the $k \times k$ diagonal matrix with diagonal entries given by a $k$-tuple $\rho = (\rho_i)_{i=1}^k \subset (0,1]^k$, and $N = 2k$ if $c=-1$.  Then $F^{-1} = cF$, and we easily compute from  \eqref{eqn:defining} that \begin{align}\label{eqn:adjoint}
u_{ij}^* = (cFUF)_{ij} =c \sum_{1 \le r,s \le N} F_{ir}F_{sj}u_{rs} = \left\{ \begin{matrix}
cF_{i,i+k}F_{j+k,j}u_{i+k,j+k} & 1 \le i,j \le 2k \\
F_{i,i+k}u_{i+k,j} & 1 \le i \le 2k, \ j > 2k \\
F_{j+k,j} u_{i,j+k} & 1 \le j \le 2k, \ i > 2k \\
u_{ij} & i,j > 2k  
\end{matrix}\right., \end{align} where in the above equations we perform the additions $i \mapsto i+k,j \mapsto j+k$ modulo $2k$.  Using these equations, it is easy to see that the fundamental representation $U = [u_{ij}]_{1 \le i,j \le N}$ of $O^+_F$ admits the following canonical block-matrix decomposition.
\begin{align}\label{eqn:U-decomp}
U =  \left(\begin{matrix}[u_{a,b}]_{1 \le a,b \le k} & [u_{a, b+k}]_{1 \le a,b \le k}  &[u_{a, t}]_{\substack{1 \le a \le k \\ t \ge 2k+1}} \\
[c \rho_a^{-1}\rho_b^{-1} u_{a,b+k}^*]_{1 \le a,b \le k} &[\rho_a^{-1}\rho_b u_{a,b}^*]_{1 \le a,b \le k}& [\rho_a^{-1}u_{a, t}^*]_{\substack{1 \le a \le k \\ t \ge 2k+1}} \\
[u_{s, b}]_{\substack{1 \le b \le k \\ s \ge 2k+1}} & [\rho_bu_{s, b}^*]_{\substack{1 \le b \le k \\ s \ge 2k+1}} & [u_{s,t}]_{s,t \ge 2k+1}\end{matrix}\right).
\end{align}

\begin{rem} \label{rem:subset-generators}
From equation \eqref{eqn:U-decomp}, we see that the C$^\ast$-algebra $C(O^+_F)$ is generated  by the subset \[\Big(\{u_{ij}\}_{\substack{1 \le i \le k \\ 1 \le j \le N}} \cup \{u_{ij}\}_{\substack{1 \le j \le k \\ 2k+1 \le i \le N}}\cup \{u_{ij}\}_{2k+1 \le i,j \le N}\Big) \subset \{u_{ij}\}_{1 \le i,j \le N}.\]
\end{rem}

\subsubsection{General $\ast$-moments over $O^+_F$}

Let $\epsilon \in \{1, \ast\}$ and $i, j \in [N]$.  Using \eqref{eqn:U-decomp} (or equivalently \eqref{eqn:adjoint}), we can find unique numbers $i_\epsilon, j_\epsilon \in [N]$ and $t_F(i,j,\epsilon) \in \R$ such that 
\begin{align} \label{eqn:t_F}
u_{ij}^{\epsilon} = t_F(i,j,\epsilon)u_{i_{\epsilon}j_{\epsilon}}.
\end{align}

In particular, the computation of arbitrary $\ast$-moments in the generators $\{u_{ij}\}_{1 \le i,j \le N}$ can be computed using the formula of Theorem \ref{thm:moments}. 

\begin{prop}\label{prop:starmoments}
Let $l \in \N$, $i,j:[l] \to [N]$ and $\epsilon:[l]\to \{1,\ast\}$ be given.  If $l$ is odd, then 
\[h_{O^+_F}\big(u_{i(1)j(1)}^{\epsilon(1)}u_{i(2)j(2)}^{\epsilon(2)}\ldots u_{i(l)j(l)}^{\epsilon(l)}\big) = 0.\]  If $l$ is even, then 
\begin{align*}h_{O^+_F}\big(u_{i(1)j(1)}^{\epsilon(1)}u_{i(2)j(2)}^{\epsilon(2)}\ldots u_{i(l)j(l)}^{\epsilon(l)}\big)&=\prod_{r=1}^lt_F(i(r),j(r), \epsilon(r)) h_{O^+_F}\Big(\prod_{r=1}^lu_{i(r)_{\epsilon(r)},j(r)_{\epsilon(r)}}\Big)\\
&=  \prod_{r=1}^lt_F(i(r),j(r), \epsilon(r)) \sum_{\pi,\sigma \in NC_2(l)} W_{l,F}(\pi,\sigma)\overline{\delta_{\pi}^F(j_\epsilon)}\delta_{\sigma}^F(i_\epsilon),
\end{align*}
where $i_\epsilon = (i_{\epsilon(1)}(1), i_{\epsilon(2)}(2), \ldots,  i_{\epsilon(l)}(l))$ and $j_\epsilon = (j_{\epsilon(1)}(1), j_{\epsilon(2)}(2), \ldots,  j_{\epsilon(l)}(l))$. 
\end{prop}

The proof of this result is immediate.

\subsubsection{Variances of the generators of $C(O^+_F)$}

The simplest (non-zero) $\ast$-moments are the left and right covariances of the generators $\{u_{ij}\}_{1 \le i,j \le N}\subset C(O^+_F)$, i.e., the quantities $\langle u_{ij} |u_{kl} \rangle_{L} = h_{O^+_F}(u_{ij}^*u_{kl})$ and $\langle u_{kl} |u_{ij} \rangle_{R}=h_{O^+_F}(u_{ij}u_{kl}^*)$.  The left and right covariances can be computed using Proposition \ref{prop:starmoments}.  Alternatively, we can compute these quantities using the Schur orthogonality relations
\begin{align} \label{eqn:Schur}
 h_{O^+_F}(u_{ij}^*u_{kl}) = \frac{\delta_{jl}(Q^{-1})_{ki}}{N_F}, \quad  h_{O^+_F}(u_{ij}u_{kl}^*) = \frac{\delta_{ik}Q_{lj}}{N_F} \qquad (1 \le i,j \le N),
\end{align}
where $Q = F^t \overline{F}$, $Q^{-1} = FF^*$ and  $\text{Tr}(Q) = \text{Tr}(Q^{-1}) = N_F$.  See for example \cite{Wo2} and the paragraphs following Theorem 7.2 of \cite{VaVe}.
In particular, when $F= F^{(c)}_\rho$ is a canonical $F$-matrix as above, then structure of $Q$ is simple: 
\[
Q = \left(\begin{matrix} D_k(\rho)^{-2}& 0 &0 \\
0 &D_k(\rho)^{2} &0 \\
0 & 0 & 1_{N-2k}\end{matrix}\right).\]
In particular, it follows from \eqref{eqn:Schur} that $\{u_{ij}\}_{1 \le i,j \le N}$ is an orthogonal system with respect to the inner products $\langle \cdot | \cdot \rangle_{L}$ and $\langle \cdot | \cdot \rangle_{R}$ induced by $h_{O^+_F}$, and a simple calculation shows that the $N \times N$ matrix of left and right variances $\Phi= \big[(\langle u_{ij} |u_{ij} \rangle_{L},\langle u_{ij} |u_{ij} \rangle_{R}) \big]_{1 \le i,j \le N}$ has the following block-matrix decomposition (compare with the decomposition of the fundamental representation of $O^+_F$ given by \eqref{eqn:U-decomp}).

\begin{align} \label{eqn:variances}
\Phi=\frac{1}{N_F}\big[(Q^{-1}_{ii}, Q_{jj}) \big]_{1 \le i,j \le N} = \frac{1}{N_F}
\left(\begin{matrix}
 [(\rho_a^2,\rho_b^{-2} )]_{1 \le a,b \le k} & [(\rho_a^2,\rho_b^{2})]_{1 \le a,b \le k} &  [(\rho_a^2,1)]_{k \times (N-2k)}  \\
 [(\rho_a^{-2},\rho_b^{-2})]_{1 \le a,b \le k}  & [(\rho_a^{-2},\rho_b^{2})]_{1 \le a,b \le k} & [(\rho_a^{-2},1)]_{k \times (N-2k)} \\ 
[(1,\rho_b^{-2})]_{(N-2k) \times k} & [(1,\rho_b^{2})]_{(N-2k) \times k} & [(1,1)]_{k \times k}
\end{matrix}\right)
\end{align}

\section{Large (quantum) dimension asymptotics}  

Using our Weingarten formulas (Theorem \ref{thm:moments} and Proposition \ref{prop:starmoments}), we can study their large (quantum) dimension asymptotics.

Let $F \in GL_N(\C)$ be such that $F\bar F = c1$, and let $N_F = \text{Tr}(F^*F)$.  We will call the number $N_F$ the {\it quantum dimension} of $O^+_F$.  Note that we always have $N_F \ge N$.  The following proposition shows that the Weingarten matrices $W_{l,F}$ associated to  $O^+_F$ are asymptotically diagonal as the quantum dimension $N_F$  tends to infinity.  This result should be compared with Theorem 6.1 in \cite{BaCo}.

\begin{thm} \label{thm:Weingartenasymptotics}
For each $l \in 2\N$, we have (as $N_F \to \infty$) that
\[
N_F^{l/2}W_{l,F}(\pi,\sigma) = \delta_{\pi,\sigma} + O(N_F^{-1}) \qquad (\pi,\sigma \in NC_2(l)).
\]
\end{thm} 

\begin{proof}
According to Theorem \ref{thm:grammatrix}, we have $W_{l,F} = G_{l,F}^{-1}$ and $G_{l,F}(\pi,\sigma) = c^{l/2+|\pi \vee \sigma|}N_F^{|\pi \vee \sigma|}$.    Observe that 
$|\pi\vee\sigma| = l/2$ if and only if $\pi = \sigma$, and $|\pi\vee\sigma| \le l/2-1$ otherwise.  Therefore we have the following asymptotic formula for $G_{l,F}$ (with respect to the operator norm):
\[
G_{l,F}= N_F^{l/2}I + O(N_F^{l/2 - 1}) = N_F^{l/2}(I + O(N_F^{-1})).
\]
Write $N_F^{-l/2}G_{l,F} = I + A_F$, where $\|A_F\| \le C_{l}N_F^{-1}$ for some $C_l \ge 0$.  Then for sufficiently large $N_F$, we have the absolutely convergent power series expansion 
\[
N_F^{l/2}W_{l,F} = (I+A_F)^{-1} = \sum_{r = 0}^\infty (-1)^rA_F^{r} = I + O(N_F^{-1}) \qquad (N_F \to \infty).
\]
The result now follows.
\end{proof}

The following proposition is a consequence of Theorem \ref{thm:Weingartenasymptotics} and gives an asymptotic factorization of the normalized joint moments over $O^+_F$. 

\begin{prop} \label{prop:asymptotic_factorization}
Fix $l \in 2\N$ and $i, j:[l]\to [N]$.  Then there is a constant $D_l > 0$ (depending only on $l$) such that
\begin{align*}&N_F^{l/2}\Big|h_{O^+_F}(u_{i(1)j(1)}u_{i(2)j(2)} \ldots u_{i(l)j(l)}) - \sum_{\pi \in NC_2(l)} \prod_{(s,t) \in \pi} h_{O^+_F}(u_{i(s)j(s)}u_{i(t)j(t)})\Big|\\
& \le \frac{D_l\max_{\pi, \sigma \in NC_2(l)}|\delta_\pi^F(j)\delta_\sigma^F(i)|}{N_F}
\end{align*}
\end{prop}

\begin{proof}
Using Theorem \ref{thm:Weingartenasymptotics}, one can find a constant $D_l > 0$ such that 
\[\sum_{\pi,\sigma \in NC_2(l)} |N_F^{l/2}W_{l,F}(\pi,\sigma) - \delta_{\pi,\sigma}| \le \frac{D_l}{N_F}.\]   Since it also follows from Theorem \ref{thm:moments} that
\begin{align*}\sum_{\pi \in NC_2(l)} \prod_{(s,t) \in \pi} h_{O^+_F}(u_{i(s)j(s)}u_{i(t)j(t)}) \
&= \sum_{\pi \in NC_2(l)} \prod_{(s,t) \in \pi} N_F^{-1} F_{i(t)i(s)}\overline{F_{j(t)j(s)}} \\
&=N_F^{-l/2} \sum_{\pi \in NC_2(l)}\delta_{\pi}^F(i) \overline{\delta_\pi^F(j)},
\end{align*} 
we obtain
\begin{align*}
&N_F^{l/2}\Big|h_{O^+_F}(u_{i(1)j(1)}u_{i(2)j(2)} \ldots u_{i(l)j(l)}) - \sum_{\pi \in NC_2(l)} \prod_{(s,t) \in \pi} h_{O^+_F}(u_{i(s)j(s)}u_{i(t)j(t)})\Big|\\ 
&=\Big|\sum_{\pi,\sigma \in NC_2(l)} N_F^{l/2}(W_{l,F}(\pi,\sigma) -\delta_{\pi,\sigma})\overline{\delta^F_\pi(j)} \delta^F_\sigma(i)\Big| \\
& \le \frac{D_l\max_{\pi, \sigma \in NC_2(l)}|\delta_\pi^F(j)\delta_\sigma^F(i)|}{N_F}.
\end{align*}
\end{proof}

Using Propositions  \ref{prop:starmoments} and \ref{prop:asymptotic_factorization}, we obtain a similar asymptotic factorization result for $\ast$-moments.  

\begin{cor} \label{cor:asymptotic_factorization_star}
Fix $l \in 2\N$, $\epsilon:[l] \to \{1,\ast\}$ and $i, j:[l]\to [N]$.  Then there is a constant $D_l > 0$ (depending only on $l$) such that
\begin{align}\label{asfact}\begin{split}&N_F^{l/2}\Big|h_{O^+_F}(u_{i(1)j(1)}^{\epsilon(1)}\ldots u_{i(l)j(l)}^{\epsilon(l)}) - \sum_{\pi \in NC_2(l)} \prod_{(s,t) \in \pi} h_{O^+_F}(u_{i(s)j(s)}^{\epsilon(s)}u_{i(t)j(t)}^{\epsilon(t)})\Big|\\
 & \le \frac{D_l\max_{\pi, \sigma \in NC_2(l)}| \delta_\pi^F(j_\epsilon)\delta_\sigma^F(i_\epsilon)| \prod_{r=1}^l|t_F(i(r),j(r), \epsilon(r))|}{N_F} .
\end{split}
\end{align}
\end{cor} 

\section{Asymptotic freeness in $O^+_F$} \label{section:freeness}

We now arrive at the main asymptotic freeness results of this paper.  

Let us fix a canonical matrix $F = F^{(c)}_\rho = \left(\begin{matrix}0 & D_k(\rho) &0 \\
cD_k(\rho)^{-1} &0 &0 \\
0 & 0 & 1_{N-2k}\end{matrix}\right) \in \text{GL}_N(\C)$ as in Section \ref{section:integrate}.
Recall from Remark \ref{rem:subset-generators} that the subset of (rescaled) matrix elements \[\mc S_F = \{\sqrt{N_F}u_{ij}\}_{\substack{1 \le i \le k \\ 1 \le j \le N}} \cup \{\sqrt{N_F}u_{ij}\}_{\substack{1 \le j \le k \\ 2k+1 \le i \le N}}\cup \{\sqrt{N_F}u_{ij}\}_{2k+1 \le i,j \le N}\]  generates the C$^\ast$-algebra $C(O^+_F)$.  In this section we show that the set $\mc S_F$ is asymptotically free in the following sense.

\begin{thm} \label{thm:asymptotic-freeness-general}
Let $\mc S = \{y_{ij}\}_{\substack{1 \le i \le k \\ 1 \le j \le N}} \cup \{y_{ij}\}_{\substack{1 \le j \le k \\ 2k+1 \le i \le N}}\cup \{y_{ij}\}_{2k+1 \le i,j \le N}$ be a family of non-commutative random variables in a NCPS $ (A, \varphi)$ with the following properties.
\begin{enumerate}
\item $\mc S$ is freely independent.
\item For each $i,j$, the elements $y_{ij} \in \mc S$ and  $\sqrt{N_F}u_{ij} \in \mc S_F$ have  the same left and right variances, given by the matrix entries of $N_F\Phi$ in equation \eqref{eqn:variances}. 
\item If either $i \le k$ or $j \le k$, then each $y_{ij}$ is a generalized circular element.
\item If $2k+1 \le i,j \le N$, then $y_{ij}$ is a standard semicircular element. 
\end{enumerate}
Then for each  $l \in 2\N$, there is a constant $D_l > 0$ such that 
\begin{align*}&\Big|h_{O^+_F}(\sqrt{N_F}u_{i(1)j(1)}^{\epsilon(1)}\ldots \sqrt{N_F}u_{i(l)j(l)}^{\epsilon(l)}) - \varphi\big(y_{i(1)j(1)}^{\epsilon(1)}\ldots y_{i(l)j(l)}^{\epsilon(l)}\big)\Big|\\
 & \le \frac{D_l\max_{\pi, \sigma \in NC_2(l)}| \delta_\pi^F(j_\epsilon)\delta_\sigma^F(i_\epsilon)| \prod_{r=1}^l|t_F(i(r),j(r), \epsilon(r))|}{N_F}, 
\end{align*}
for each $\epsilon:[l] \to \{1,\ast\}$ and $i, j:[l]\to [N]$. 
\end{thm}    

\begin{proof}
Since $\mc S$ is a free family consisting of generalized circular elements and standard semicircular elements satisfying conditions (2)-(4) above, Theorem \ref{thm:mc-formulas} gives
\begin{align*}
\varphi\big(y_{i(1)j(1)}^{\epsilon(1)}\ldots y_{i(l)j(l)}^{\epsilon(l)}\big) &= \sum_{\pi \in NC_2(l)} \prod_{(s,t) \in \pi} \varphi(y_{i(s)j(s)}^{\epsilon(s)}y_{i(t)j(t)}^{\epsilon(t)})\\
&=\sum_{\pi \in NC_2(l)} \prod_{(s,t) \in \pi}  N_Fh_{O^+_F}(u_{i(s)j(s)}^{\epsilon(s)}u_{i(t)j(t)}^{\epsilon(t)}).
\end{align*} 
The theorem now follows from Corollary \ref{cor:asymptotic_factorization_star}.
\end{proof}

\subsection{Asymptotic freeness in the large dimension limit}

Using Theorem \ref{thm:asymptotic-freeness-general}, we see that the quantum groups  $O^+_F$ provide asymptotic models for almost-periodic free Araki-Woods factors. That is, canonical generators of almost-periodic free Araki-Woods factors can be approximated in joint distribution by normalized coordinates over a suitable sequence of $O^+_F$ quantum groups.

To see this, let $\Gamma(H_\R,U_t)''$ be an almost-periodic free Araki-Woods factor.  Then we can write $\Gamma(H_\R, U_t)'' = (z_i : i \in \N)''$, where $(z_i)_{i \in \N}$ is a free family
of $(1, \lambda_i)$-generalized circular elements $z_i$ (with $1 < \lambda_i < \infty$).  To approximate the joint $\ast$-distribution of $(z_i)_{i \in \N}$, define, for each $k \in \N$, \[F(k)  = \left(\begin{matrix}0 &D_{k+1}(1,\sqrt{\lambda_1}, \ldots, \sqrt{\lambda_k})^{-1}  \\
-D_{k+1}(1,\sqrt{\lambda_1}, \ldots, \sqrt{\lambda_k})  &0 \end{matrix}\right) \in \text{GL}_{2k+2}(\C).\]

\begin{thm} \label{thm:asymptotic-freeness-large-rank}
The family of non-commutative random variables $(z_{i,k})_{i=1}^k = (\sqrt{N_{F(k)}}u_{1,i+1})_{i = 1}^k \subset (C(O^+_{F(k)}), h_{O^+_{F(k)}})$ converges in joint distribution as $k \to \infty$ to $(z_i)_{i \in \N} \subset (\Gamma(H_\R,U_t)'', \varphi_\Omega)$.    
\end{thm}

\begin{proof}  By construction, $(z_{i,k})_{i = 1}^k$ and $(z_i)_{i = 1}^k$ have the same left and right variances.  By Theorem \ref{thm:asymptotic-freeness-general}, we then have for any $l \in 2\N$, $\epsilon:[l]\to \{1,\ast\}$, $i:[l] \to \N$, and $k = k(i)$ sufficiently large,
\begin{align*}&\Big|h_{O^+_{F(k)}}(z_{i(1),k}^{\epsilon(1)}\ldots z_{i(l),k}^{\epsilon(l)}) - \varphi_\Omega\big(z_{i(1)}^{\epsilon(1)}\ldots z_{i(l)}^{\epsilon(l)}\big)\Big|\\
 & \le \frac{D_l\max_{\pi, \sigma \in NC_2(l)}| \delta_\pi^{F(k)}(1_\epsilon)\delta_\sigma^{F(k)}((i+1)_\epsilon)| \prod_{r=1}^l|t_{F(k)}(1,i(r)+1, \epsilon(r))|}{N_{F(k)}}.
\end{align*}
Since both quantities $\max_{\pi, \sigma \in NC_2(l)}| \delta_\pi^{F(k)}(1_\epsilon)\delta_\sigma^{F(k)}((i+1)_\epsilon)|$ and $\prod_{r=1}^l|t_{F(k)}(1,i(r)+1, \epsilon(r))|$ are constant once $l,i$ and $\epsilon$ are fixed, and $N_{F(k)}  = \Tr(F(k)^*F(k)) \ge \Tr(1) =2k+2$, we conclude that
\[
 \Big|h_{O^+_{F(k)}}(z_{i(1),k}^{\epsilon(1)}\ldots z_{i(l),k}^{\epsilon(l)}) - \varphi_\Omega\big(z_{i(1)}^{\epsilon(1)}\ldots z_{i(l)}^{\epsilon(l)}\big)\Big| \le \frac{\text{Constant}}{2k+2} \to 0.
\]
\end{proof}

\begin{rem} \label{rem:all-variables} 
Using the same reasoning as in the proof of Theorem \ref{thm:asymptotic-freeness-large-rank}, it is easy to see that for the above sequence of quantum groups $(O^+_{F(k)})_{k \in \N}$, the entire family of normalized generators \[\mc S_{F(k)} = \{\sqrt{N_{F(k)}}u_{ij}\}_{1 \le i,j \le k} \cup \{\sqrt{N_{F(k)}}u_{i,j+k}\}_{1 \le i,j \le k} \] of $C(O^+_{F(k)})$ converges in distribution to a free family of generalized circular elements $\{y_{ij}\}_{1 \le i,j < \infty} \cup \{w_{ij}\}_{1 \le i,j < \infty}$ in a NCPS $(A,\varphi)$ with the following left and right variances (determined by Theorem \ref{thm:asymptotic-freeness-general}):
\[\varphi(y_{ij}^*y_{ij}) = \rho_i^{2}, \quad \varphi(y_{ij}y_{ij}^*) = \rho_j^{-2} \quad \varphi(w_{ij}^*w_{ij}) = \rho_i^2, \quad \varphi(w_{ij}w_{ij}^*) = \rho_j^{2}, 
\]
where $\rho_1 = 1$ and $\rho_i = \lambda_{i-1}^{-1/2}$ for $i \ge 2$.  Note also  that there is a state-preserving $\ast$-isomorphism $W^*(\{y_{ij}\}_{1 \le i,j < \infty} \cup \{w_{ij}\}_{1 \le i,j < \infty})$ and $W^*((z_i)_{i \in \N}) = \Gamma(H_\R, U_t)'' $, the almost-periodic free Araki-Woods factor we started with. (This isomorphism follows from \cite[Theorem 6.4]{Sh}). 
\end{rem}

\subsection{Asymptotic freeness in finite dimensions}
In Theorem \ref{thm:asymptotic-freeness-large-rank}, we saw that normalized generators of a suitable family of $O^+_F$'s  converge in distribution to free random variables as the size $N$ of the matrices $F \in \text{GL}_N(\C)$ go to infinity.  On the other hand, the general estimate of Theorem \ref{thm:asymptotic-freeness-general} shows that in the non-tracial setting, the rate of approximation to freeness is governed by the growth of the quantum dimension $N_F = \Tr(F^*F)$, and not the physical dimension $N$.  This new phenomenon allows one to consider scenarios where $N_F \to \infty$, while the dimension $N$ of $F \in \text{GL}_N(\C)$ is fixed.  This is illustrated by the next theorem. 

\begin{thm} \label{thm:asymptotic-freeness-large-qrank}
Fix $k \in \N$ and a sequence $\rho = (\rho_1, \ldots, \rho_k) \in (0,1)^k$ and let 
\[F(\gamma) = \left(\begin{matrix}0 &D_{k+1}(\rho)&0&0 \\
D_{k+1}(\rho)^{-1}  &0&0&0 \\
0&0&0&\gamma \\
 0&0&\gamma^{-1}&0\end{matrix}\right) \in \text{GL}_{2k+2}(\C)  \qquad (0 < \gamma < 1).\]  Then the subset of generators \[\tilde{\mc S}_{F(\gamma)} = \{\sqrt{N_{F(\gamma)}}u_{ij}\}_{\substack{1 \le i \le k \\ 1 \le j \le 2k}}\subset (C(O^+_{F(\gamma)}), h_{O^+_{F(\gamma)}})\] converges in distribution to a free
family of generalized circular elements \[\tilde{\mc S} = \{y_{ij}\}_{\substack{1 \le i \le k \\ 1 \le j \le 2k}}\] in a NCPS $(A,\varphi)$ with left and right variances given by 
\[\varphi(y_{ij}^*y_{ij}) = \rho_i^{2}, \quad \varphi(y_{ij}y_{ij}^*) = \Bigg\{ \begin{matrix} \rho_j^{-2} & 1 \le j \le k \\
\rho_j^2 & k+1 \le j \le 2k\end{matrix}. \] 
\end{thm}

\begin{rem} \label{rem:nolimit} Note that Theorem \ref{thm:asymptotic-freeness-large-qrank}, makes a statement about the limiting distribution of a subset $\tilde{\mc S}_{F(\gamma)}$ of generators of $C(O^+_{F(\gamma)})$.  We cannot make a statement about the asymptotic freeness of the entire family of generators $\mc S_{F(\gamma)}$ in this setting since some of these variables do not have limiting distributions.  For example, 
\[ 
h_{O^+_{F(\gamma)}}(\sqrt{N_{F(\gamma)}}u_{k+1,k+1}^{\epsilon(1)}\sqrt{N_{F(\gamma)}}u_{k+1,k+1}^{\epsilon(2)}) = \Bigg\{ \begin{matrix} \gamma^2 & \epsilon(1) = \ast, \ \epsilon(2) = 1 \\
\gamma^{-2} & \epsilon(1) = 1, \ \epsilon(2) = \ast,
\end{matrix}
\]
which implies that $\sqrt{N_{F(\gamma)}}u_{k+1,k+1}$ does not have a limiting distribution as $\gamma \to 0$.
\end{rem}

\begin{proof}[Proof of Theorem \ref{thm:asymptotic-freeness-large-qrank}]
By construction, the families $\tilde{\mc S}$ and $\tilde{\mc S}_{F(\gamma)}$ have the same left and right variances (which are independent of $\gamma$).  Therefore we may apply Theorem \ref{thm:asymptotic-freeness-general} to obtain, for any $l \in 2\N$, $\epsilon:[l]\to \{1,\ast\}$, $i:[l] \to [k]$, $j:[l] \to [2k]$,
\begin{align*}
&\Big|h_{O^+_{F(\gamma)}}(\sqrt{N_{F(\gamma)}}u_{i(1)j(1)}^{\epsilon(1)}\ldots \sqrt{N_{F(\gamma)}}u_{i(l)j(l)}^{\epsilon(l)}) - \varphi\big(y_{i(1)j(1)}^{\epsilon(1)}\ldots y_{i(l)j(l)}^{\epsilon(l)}\big)\Big|\\
 & \le \frac{D_l\max_{\pi, \sigma \in NC_2(l)}| \delta_\pi^{F(\gamma)}(j_\epsilon)\delta_\sigma^{F(\gamma)}(i_\epsilon)| \prod_{r=1}^l|t_{F(\gamma)}(i(r),j(r), \epsilon(r))|}{N_{F(\gamma)}}. 
\end{align*}
Since the numerator of the above expression is fixed with respect to $\gamma$ and $N_{F(\gamma)} =\gamma^2 + \gamma^{-2} + \sum_{i=1}^k (\rho_i^2+\rho_i^{-2}) \to \infty$ as $\gamma \to 0$, the theorem follows.
\end{proof}

\subsection{A remark on the free unitary case} \label{sect:unitary} 

Let $c = \pm 1$ and $F \in \text{GL}_N(\C)$ be a canonical $F$-matrix.  In this section we briefly comment on the free unitary quantum groups $U^+_F$.  

In this case, it is known from the fundamental work of Banica \cite{Ba} that there is an injective $\ast$-homomorphism $C(U_F^+)\hookrightarrow C(\T)*C(O^+_F)$ (the unital free product of $C(\T)$ and $C(O^+_F)$) given by $v_{ij} \mapsto wu_{ij}$.  Here, $w \in C(\T)$ is canonical unitary coordinate function on the unit circle $\T$.  Moreover, it is known that under the above embedding, $h_{U_F^+} = (\tau * h_{O^+_F})|_{C(U^+_F)}$, where $\tau$ denotes integration with respect to the Haar probability measure on $\T$.

In other words, the variables $\{v_{ij}\}_{1 \le i,j \le N} \subset (C(U^+_F), h_{U^+_F})$  and $\{wu_{ij}\}_{1 \le i,j \le N} \subset (C(\T)*C(O^+_F),\tau * h_{O^+_F})$ are identically distributed.  Using this observation, together with some basic facts about free independence and the results we have already obtained on $O^+_F$, we arrive at the following unitary version of Theorem \ref{thm:asymptotic-freeness-general},  whose proof we leave as an exercise to the reader.  (Note that the extra freeness inside $C(U^+_F)$ given by the above free product model  yields a slightly cleaner statement than Theorem \ref{thm:asymptotic-freeness-general}.)  

\begin{thm}\label{thm:asymptotic-freeness-general-unitary}
Fix a canonical deformation matrix \[F = F^{(c)}_\rho = \left(\begin{matrix}0 & D_k(\rho) &0 \\
cD_k(\rho)^{-1} &0 &0 \\
0 & 0 & 1_{N-2k}\end{matrix}\right) \in \text{GL}_N(\C),\] and let $\mc S = \{y_{ij}\}_{1 \le i,j \le N}$ be a family of non-commutative random variables in a NCPS $ (A, \varphi)$ with the following properties.
\begin{enumerate}
\item $\mc S$ is freely independent.
\item For each $i,j$, the elements $y_{ij} \in \mc S$ and  $\sqrt{N_F}v_{ij} \in C(U_F^+)$ have  the same left and right variances, given by the matrix entries of $N_F\Phi$ in equation \eqref{eqn:variances}. 
\item Each $y_{ij}$ is a generalized circular element.
\end{enumerate}
Then for each  $l \in 2\N$, there is a constant $D_l > 0$ such that 
\begin{align*}&\Big|h_{O^+_F}(\sqrt{N_F}v_{i(1)j(1)}^{\epsilon(1)}\ldots \sqrt{N_F}v_{i(l)j(l)}^{\epsilon(l)}) - \varphi\big(y_{i(1)j(1)}^{\epsilon(1)}\ldots y_{i(l)j(l)}^{\epsilon(l)}\big)\Big|\\
 & \le \frac{D_l\max_{\pi, \sigma \in NC_2(l)}| \delta_\pi^F(j_\epsilon)\delta_\sigma^F(i_\epsilon)| \prod_{r=1}^l|t_F(i(r),j(r), \epsilon(r))|}{N_F}, 
\end{align*}
for each $\epsilon:[l] \to \{1,\ast\}$ and $i, j:[l]\to [N]$. 
\end{thm}    

\section{$O^+_F$ and quantum symmetries of free Araki-Woods factors} \label{qsym}

In this final section we return to the free orthogonal quantum groups $O^+_F$ and investigate to what extent they can be regarded as quantum symmetries of free Araki-Woods factors.  Inspired by the fact that the free group factors $L(\F_N)$ admit $O_N^+$ as natural quantum symmetries \cite{Cu}, we are interested in finding canonical families of (non-tracial) non-commutative random variables $(x_1, \ldots, x_N)$ belonging to a 
NCPS $(A,\varphi)$ whose joint distribution is $O^+_F$-invariant in the following sense.

\begin{defn}
Let $(A,\varphi)$ be a NCPS and consider an $N$-tuple ${\bf x} = (x_1, \ldots, x_N) \subset A$.  Fix $F \in \text{GL}_N(\C)$ and let $U = [u_{ij}] \in M_N(C(O^+_F))$ be the fundamental representation of $O^+_F$.  We say that \textit{${\bf x}$ has an $O^+_F$-invariant joint distribution} (or, that ${\bf x}$ is \textit{$O^+_F$-rotatable}) if for any $l \in \N$, $i:[l] \to [N]$
and any $\epsilon:[l] \to \{1,\ast\}$, we have
\begin{align}\label{eqn:invariance} 
\sum_{j:[l]\to N} u_{i(1)j(1)}^{\epsilon(1)}u_{i(2)j(2)}^{\epsilon(2)} \ldots u_{i(l)j(l)}^{\epsilon(l)} \varphi\big(x_{j(1)}^{\epsilon(1)}x_{j(2)}^{\epsilon(2)} \ldots x_{j(l)}^{\epsilon(l)}\big)   = \varphi \big(x_{i(1)}^{\epsilon(1)}x_{i(2)}^{\epsilon(2)} \ldots x_{i(l)}^{\epsilon(l)}\big)1.
\end{align}
\end{defn}
 
We note that the existence of an $N$-tuple ${\bf x}$ with the above $O^+_F$-invariance property is connected to the existence of an \textit{action} of the quantum group $O^+_F$ on the von Neumann algebra generated by ${\bf x}$.

\begin{defn}\label{defn:action}
Let $\G$ be a compact quantum group with von Neummann algebra $L_\infty(\G) = \pi_h(C(\G))''$
and (extended) coproduct $\Delta_r:L_\infty(\G) \to L_\infty(\G) \overline{\otimes} L_\infty(\G)$.  Let $M$ a von Neumann algebra. 
\begin{enumerate} 
\item A \textit{left action} (or simply an \textit{action}) of $\G$ on $M$ is a normal, injective and unital $\ast$-homomorphism $\alpha: M \to L_\infty(\G) \overline{\otimes} M$ such that $(\iota \otimes \alpha) \circ\alpha = (\Delta \otimes \iota) \circ \alpha$.     We denote the action of $\G$ on $M$ by the notation $\G \curvearrowright^\alpha M$. 
\item If $\varphi$ is a faithful normal state on $M$, we say that an action $\G \curvearrowright^\alpha M$ is \textit{$\varphi$-preserving} if $(\iota \otimes \varphi) \circ \alpha = \varphi(\cdot)1_{L_\infty(\G)}$.  Such an action has notation $\G \curvearrowright^\alpha (M,\varphi)$. 
\end{enumerate}
\end{defn}

\begin{rem}
If ${\bf x} = (x_1, \ldots, x_N) \subset (A, \varphi)$ is an $N$-tuple with an $O^+_F$-invariant joint distribution, then it is easy to see that 
\[\alpha(x_i) = \sum_{j=1}^N \pi_h(u_{ij}) \otimes x_j \qquad (1 \le i \le n)
\] 
defines a $\varphi$-preserving action $O^+_F \curvearrowright^{\alpha} (W^*(x_1, \ldots, x_N) , \varphi)$, where $\pi_h:C(O^+_F) \to L^\infty(O^+_F)$ is the GNS representation associated to the Haar state. 
\end{rem}

We now show that such $O^+_F$-invariant non-commutative random variables exist for any $F$.
To do this, we first need a lemma. 

\begin{lem} \label{lem:Q-relation}Let $Q = F^t\bar{F}$, where $F \in \text{GL}_N(\C)$ is a canonical $F$ matrix (so that, in particular, $Q$ is diagonal), and let $U =[u_{ij}] \in M_N (C(O^+_F))$ be the fundamental representation of $O^+_F$.  Then  \[\sum_{r = 1}^N u_{ir}u_{jr}^* = \delta_{ij}1\quad  \text{and} \quad \sum_{r=1}^N u_{ir}^*u_{jr}(Q^{-1})_{rr} = \delta_{ij}(Q^{-1})_{ii}1 \qquad (1 \le i,j \le N).\]  
\end{lem}

\begin{proof}
The first equality is a direct consequence of the fact that the fundamental representation $U$ is unitary.  To prove the second inequality, we use \cite[Proposition 3.2.17]{Ti} which shows that 
$\bar{U} \overline{Q}^{-1}U^t = \overline{Q}^{-1}$.  Since $Q$ is diagonal and positive definite, the result follows.
\end{proof}

In the case where $F \in \text{GL}_N(\C)$ is canonical, the way to find $O^+_F$-rotatable generators of a free Araki-Woods factor is to fix a column of the fundamental representation of $O^+_F$ and take an $N$-tuple of freely independent generalized circular elements with the same left and right variances as this column (up to a common non-zero scaling factor).  Again,
we note that after completion of a first draft of this paper, it was pointed out to the authors that the following theorem can also be obtained as a consequence of \cite[Proposition 3.1]{Va05}.      

\begin{thm} \label{thm:invariance-canonicalF}
Let $F \in \text{GL}_N(\C)$ be a canonical $F$ matrix and let $Q = F^t\bar{F}$ with diagonal entries $(Q_{ii})_{i=1}^N$.  Let ${\bf x} = (x_1, \ldots, x_N) \subset (A, \varphi)$ be a $\ast$-free family of generalized circular elements with left and right covariances given by 
\[
\varphi(x_i^*x_i) = Q_{ii}^{-1}, \quad \varphi(x_ix_i^*) = 1, \qquad (1 \le i \le N). 
\]
Then ${\bf x}$ has an $O^+_F$-invariant joint $\ast$-distribution.  In other words, there is a $\varphi$-preserving action $O^+_F \curvearrowright^\alpha(W^*(x_1, \ldots, x_N), \varphi)$
given by $\alpha(x_i) = \sum_{j=1}^N\pi_h(u_{ij}) \otimes x_j$, where $U = [u_{ij}]$ is the fundamental representation of $O^+_F$.   
\end{thm}

\begin{proof}
We must verify \eqref{eqn:invariance} for the $N$-tuple ${\bf x}$, for each choice of $l \in \N$, $i:[l] \to [N]$, and $\epsilon:[l] \to \{1,\ast\}$.  To start, observe that when $l$ is odd or $|\epsilon^{-1}(1)| \ne |\epsilon^{-1}(\ast)|$, then both sides of \eqref{eqn:invariance} are always zero.  Therefore we assume  $l \in 2\N$ and  that $|\epsilon^{-1}(1)| =|\epsilon^{-1}(\ast)| = l/2$.  

We begin by considering the case $l=2$, and fix $1 \le i(1), i(2) \le N$, $\epsilon(1)\ne\epsilon(2) \in \{1,\ast\}$.  Then we have 
\begin{align*}
\sum_{j(1),j(2)=1}^N u^{\epsilon(1)}_{i(1)j(1)}u^{\epsilon(2)}_{i(2)j(2)}\varphi(x_{j(1)}^{\epsilon(1)}x_{j(2)}^{\epsilon(2)})
 &= \sum_{j(1)=1}^N u^{\epsilon(1)}_{i(1)j(1)}u^{\epsilon(2)}_{i(2)j(1)}\varphi(x_{j(1)}^{\epsilon(1)}x_{j(1)}^{\epsilon(2)}).
\end{align*}
If $\epsilon(1) = 1$ and $\epsilon(2) = \ast$, then $\varphi(x_{j(1)}x_{j(1)}^*) = 1$ and the above quantity equals 
\[
 \sum_{j(1)=1}^N u_{i(1)j(1)}u^{\ast}_{i(2)j(1)} = \delta_{i(1),i(2)}1 = \varphi(x_{i(1)}x_{i(2)}^*)1  \qquad (\text{by Lemma \ref{lem:Q-relation}}).
\]
If $\epsilon(1) = \ast$ and $\epsilon(2) = 1$, then $\varphi(x_{j(1)}^*x_{j(1)}) = Q_{j(1)j(1)}^{-1}$ and the above quantity equals  
\[
 \sum_{j(1)=1}^N u_{i(1)j(1)}^{\ast}u_{i(2)j(1)}Q_{j(1)j(1)}^{-1} = \delta_{i(1),i(2)} \varphi(x_{i(1)}^*x_{i(1)})1 =  \varphi(x_{i(1)}^*x_{i(2)})1  \qquad (\text{by Lemma \ref{lem:Q-relation}}).
\]
In each case, we obtain 
\begin{align} \label{eqn:l=2}
\sum_{j(1),j(2)=1}^N u^{\epsilon(1)}_{i(1)j(1)}u^{\epsilon(2)}_{i(2)j(2)}\varphi(x_{j(1)}^{\epsilon(1)}x_{j(2)}^{\epsilon(2)})
= \varphi(x_{i(1)}^{\epsilon(1)}x_{i(2)}^{\epsilon(2)})1.
\end{align}
Now let $2 < l \in 2\N$ and fix $\epsilon:[l] \to \{1,\ast\}$ and $i:[l] \to [N]$.  Then we have
from Theorem \ref{thm:mc-formulas}
\begin{align*}
&\sum_{j:[l]\to N} u_{i(1)j(1)}^{\epsilon(1)}u_{i(2)j(2)}^{\epsilon(2)} \ldots u_{i(l)j(l)}^{\epsilon(l)} \varphi\big(x_{j(1)}^{\epsilon(1)}x_{j(2)}^{\epsilon(2)} \ldots x_{j(l)}^{\epsilon(l)}\big)\\
&= \sum_{j:[l]\to N} u_{i(1)j(1)}^{\epsilon(1)}u_{i(2)j(2)}^{\epsilon(2)} \ldots u_{i(l)j(l)}^{\epsilon(l)} \Big(\sum_{\pi \in NC_2^\epsilon(l)} \prod_{(s,t) \in \pi} \varphi(x_{j(s)}^{\epsilon(s)}x_{j(t)}^{\epsilon(t)}) \Big)\\
&=  \sum_{\pi \in NC_2^\epsilon(l)} \Big(\sum_{j:[l]\to N} u_{i(1)j(1)}^{\epsilon(1)}u_{i(2)j(2)}^{\epsilon(2)} \ldots u_{i(l)j(l)}^{\epsilon(l)} \prod_{(s,t) \in \pi} \varphi(x_{j(s)}^{\epsilon(s)}x_{j(t)}^{\epsilon(t)})\Big).
\end{align*}  
Fix $\pi \in NC_2^\epsilon(l)$ and consider the internal sum above.  Since $\pi$ is non-crossing, it contains a neighboring pair $(r,r+1)$.  Applying \eqref{eqn:l=2} to the partial sum over $1 \le j(r),j(r+1) \le N$, we obtain
\begin{align*}
&\sum_{j:[l] \to [N]} u_{i(1)i(1)}^{\epsilon(1)}  u_{i(2)i(2)}^{\epsilon(2)} \ldots  u_{i(l)i(l)}^{\epsilon(l)} \prod_{(s,t) \in \pi} \varphi(x_{j(s)}^{\epsilon(s)}x_{j(t)}^{\epsilon(t)}) \\
&=\varphi(x_{i(r)}^{\epsilon(r)}x_{i(r+1)}^{\epsilon(r+1)})  \\ 
&\quad \times \Big(\sum_{j:[l]\backslash\{r,r+1\}\to [N]}u_{i(1)i(1)}^{\epsilon(1)} \ldots u_{i(r-1)i(r-1)}^{\epsilon(r-1)}u_{i(r+2)i(r+2)}^{\epsilon(r+2)} \dots u_{i(l)i(l)}^{\epsilon(l)}\prod_{(s,t) \in \pi\backslash (r,r+1)} \varphi(x_{j(s)}^{\epsilon(s)}x_{j(t)}^{\epsilon(t)}) \Big).
\end{align*}
Repeatedly applying the same principle to this new internal sum of lower order (note that $\pi\backslash (r,r+1)$ is again non-crossing), we obtain (after a total of $l/2-1$ steps)  
\begin{align}\label{eqn:l}
&\sum_{j:[l] \to [N]} u_{i(1)i(1)}^{\epsilon(1)}  u_{i(2)i(2)}^{\epsilon(2)} \ldots  u_{i(l)i(l)}^{\epsilon(l)} \prod_{(s,t) \in \pi} \varphi(x_{j(s)}^{\epsilon(s)}x_{j(t)}^{\epsilon(t)}) = \prod_{(s,t) \in \pi}   \varphi(x_{i(s)}^{\epsilon(s)}x_{i(t)}^{\epsilon(t)})1.
\end{align}
Therefore 
\begin{align*}
&\sum_{j:[l]\to N} u_{i(1)j(1)}^{\epsilon(1)}u_{i(2)j(2)}^{\epsilon(2)} \ldots u_{i(l)j(l)}^{\epsilon(l)} \varphi\big(x_{j(1)}^{\epsilon(1)}x_{j(2)}^{\epsilon(2)} \ldots x_{j(l)}^{\epsilon(l)}\big)\\
&= \sum_{\pi \in NC_2^\epsilon(l)} \prod_{(s,t) \in \pi}   \varphi(x_{i(s)}^{\epsilon(s)}x_{i(t)}^{\epsilon(t)})1 = \varphi(x_{i(1)}^{\epsilon(1)}x_{i(2)}^{\epsilon(2)} \ldots x_{i(l)}^{\epsilon(l)})1.
\end{align*}
\end{proof}

\begin{rem}
It is clear that $(W^*(x_1, \ldots, x_N), \varphi) \cong (\Gamma(\mc H_\R, U_t)'', \varphi_\Omega)$ is a free Araki-Woods factor associated to a finite dimensional orthogonal representation $(U_t)_{r \in \R}$  (compare with Theorem \ref{thm:repfromgenerators}).  Moreover, it is interesting to note that the type classification for the von Neumann algebras $L^\infty(O^+_F)$ and $\Gamma(H_\R, U_t)''$ is the same.  More precisely, if $\Gamma < \R^*_+$ is the subgroup generated by the eigenvalues of $Q \otimes Q^{-1}$, then both of these algebras are type II$_1$ when $Q = 1$, III$_\lambda$ if $\Gamma = \lambda^\Z$, and type III$_1$ otherwise.  Compare \cite[Theorem 6.1]{Sh} and \cite[Theorem 7.1]{VaVe}.
\end{rem}

\begin{rem} 
Theorem \ref{thm:invariance-canonicalF} only considers the case of a canonical matrix $F$.  For generic $F \in \text{GL}_N(\C)$ such that $F\bar F = c 1$, recall from Section \ref{section:canonicalF} that there is a canonical $F$-matrix $F^{(c)}_{\rho} \in \text{GL}_N(\C)$ and $v \in \mc U_N$ such that $F^{(c)}_{\rho} = vF v^t$ and $O^+_F \cong O^+_{F^{(c)}_{\rho}}$.  Then $O^+_F \curvearrowright^{\alpha_F} (\Gamma(\mc H_\R, U_t)'', \varphi_\Omega)$, where $\Gamma(\mc H_\R, U_t)''$ is the free Araki-Woods factor on which $O^+_{F^{(c)}_{\rho}}$ acts in the sense of the above theorem.  Indeed let ${\bf x} = (x_1, \ldots, x_N)$ be the generalized circular system constructed in Theorem \ref{thm:invariance-canonicalF}, let $\alpha_{F^{(c)}_{\rho}}$ be the corresponding action, and let ${\bf y} = v{\bf x}$.  Then $W^*({\bf x}) = W^*({\bf y})$ and one readily checks from the defining relations that \[W = vUv^*,\]
where $W = [w_{ij}]$ and $U = [u_{ij}]$ are the fundamental representations of $O^+_F$ and $O^+_{F^{(c)}_{\rho}}$, respectively.  A simple calculation then shows that condition \eqref{eqn:invariance} holds with the $w_{ij}$'s replacing the $u_{ij}$'s and the $y_i$'s replacing the $x_i$'s.
\end{rem}


\begin{thebibliography}{99}


\bibitem{Ba0} T. Banica, {\it Th\'eorie des repr\'esentations du groupe quantique compact libre $O(n)$},  C. R. Acad. Sci. Paris S\'er. I Math. 322 (1996), no. 3, 241--244.

\bibitem{Ba} T. Banica, {\it Le groupe quantique compact libre $U(n)$}, Comm. Math. Phys. 190 (1997), 143--172.

\bibitem{BaCo} T. Banica and B. Collins, {\it  Integration over compact quantum groups},  Publ. Res. Inst. Math. Sci. 43 (2007), 277--302.

\bibitem{BaCoZJ} T. Banica, B. Collins and P. Zinn-Justin, {\it Spectral analysis of the free orthogonal matrix},   Int. Math. Res. Notices 17 (2009), 3286-3309.

\bibitem{BiDeVa} J. Bichon, A. De Rijdt and S. Vaes, {\it Ergodic coactions with large multiplicity and monoidal equivalence of quantum groups}, Comm. Math. Phys. 262 (2006), 703--728. 

\bibitem{Co03} B. Collins, {\it Moments and Cumulants of Polynomial random variables on unitary groups, the Itzykson-Zuber integral and free probability}, Int. Math. Res. Not. 17 (2003), 953--982. 

\bibitem{CoSn06} B. Collins and Piotr Sniady, {\it Integration with respect to the Haar measure on unitary, orthogonal and symplectic group},  Comm. Math. Phys. 264 (2006), 773--795.

\bibitem{Cu}  S. Curran, {\it Quantum rotatability}, Trans. Amer. Math. Soc. 362 (2010), 4831--4851. 

\bibitem{DeFrYa} K. De Commer, A. Freslon and M. Yamashita, {\it CCAP for universal discrete quantum groups}, Comm. Math. Phys. 331 (2014), 677--701.

\bibitem{DiFr87} P. Diaconis and D. Freedman, { \it A dozen de Finetti-style results in search of a theory}, Ann. Inst. H. Poincare Probab. Statist., 23 (1987), 397--423.

\bibitem{Fr} D. Freedman, {\it Invariants under mixing which generalize de Finetti’s theorem}, Ann. Math. Stat., 33 (1962), 916--923.

\bibitem{Ka}  O. Kallenberg, {\it Probabilistic symmetries and invariance principles},  Probability and its applications, Springer-Verlag, (2005).

\bibitem{NiSp}  A. Nica and R. Speicher, {\it Lectures on the combinatorics of free probability},  London Math. Soc.
Lect. Note Ser. 335, Cambridge University Press, Cambridge 2006.

\bibitem{Sh} D. Shlyakhtenko, {\it Free quasi-free states}, Pac. Jour. Math 177 (1997), 329--368.

\bibitem{Ti}  T. Timmerman,  {\it An invitation to quantum groups and duality}, EMS Textbooks in Mathematics, Zurich 2008.

\bibitem{Va05} S. Vaes, {\it Strictly outer actions of groups and quantum groups}, 
J. reine angew. Math. 578 (2005), 147--184.

\bibitem{VaVe}  S. Vaes and R. Vergnioux, {\it The boundary of universal discrete quantum groups, exactness, and factoriality}, Duke Math. J. 140 (2007), 35--84.

\bibitem{VaWa} A. Van Daele and S. Wang, {\it Universal quantum groups}, Internat. J. Math. 7 (1996), 255--263.

\bibitem{Wa02} S. Wang, {\it Structure and isomorphism classification of compact quantum groups $A_u(Q)$ and $B_u(Q)$}, J. Operator Theory 48 (2002), 573--583. 

\bibitem{Wa0} S. Wang, {\it General constructions of compact quantum groups},  Ph.D. Thesis, U. California Berkeley, 1993.

\bibitem{We}  D. Weingarten, {\it Asymptotic behavior of group integrals in the limit of infinite rank},  J. Math. Phys.
19 (1978), 999--1001.

\bibitem{Wo}  S. Woronowicz, {\it Compact matrix pseudogroups},  Comm. Math. Phys. 111 (1987), 613--665.

\bibitem{Wo2} S. Woronowicz,  {\it Compact quantum groups},  Sym\'etries quantiques (Les Houches, 1995), North-
Holland, Amsterdam (1998),  845--884.

\end{thebibliography}
\end{document}